\documentclass{amsart}
\usepackage{amssymb, mathrsfs, graphicx, subfig}
\usepackage[german, english]{babel}
\copyrightinfo{2010}{Matthias Keller, Sebastian Haeseler, Rados{\l}aw K. Wojciechowski}

\newtheorem{lemma}{Lemma}[section]
\newtheorem{thm}[lemma]{Theorem}
\newtheorem{proposition}[lemma]{Proposition}
\newtheorem{corollary}[lemma]{Corollary}

\theoremstyle{definition}

\newtheorem*{example}{Example}

\theoremstyle{remark}
\newtheorem{remark}{Remark}

\numberwithin{equation}{section}

\newcommand{\R}{{\mathbb R}}

\newcommand{\N}{{\mathbb N}}

\newcommand{\E}{{\mathcal E}}

\newcommand{\se}{\sigma_{\!\mathrm{ess}}}

\newcommand{\Deg}{{\mathrm{Deg}}}

\newcommand{\supp}{{\mathrm {supp}\,}}

\newcommand{\al}{{\alpha}}

\newcommand{\de}{{\delta}}
\newcommand{\eps}{{\varepsilon}}

\newcommand{\ph}{{\varphi}}
\newcommand{\lm}{{\lambda}}

\newcommand{\Gm}{{\Gamma}}
\newcommand{\si}{{\sigma}}

\newcommand{\as}[1]{\left\langle #1\right\rangle}

\newcommand{\aV}[1]{\left\Vert #1\right\Vert}

\newcommand{\ow}[1]{\widetilde{ #1}}

\newcommand{\Hm}[1]{\leavevmode{\marginpar{\tiny%
$\hbox to 0mm{\hspace*{-0.5mm}$\leftarrow$\hss}%
\vcenter{\vrule depth 0.1mm height 0.1mm width \the\marginparwidth}%
\hbox to 0mm{\hss$\rightarrow$\hspace*{-0.5mm}}$\\\relax\raggedright
#1}}}

\begin{document}

\title[Volume growth and bounds for the essential spectrum]{Volume growth and bounds for the essential spectrum for Dirichlet forms}

\author{Sebastian Haeseler}
\address{Mathematisches Institut \\
Friedrich Schiller Universit{\"a}t Jena \\
07743 Jena, Germany }
\email{sebastian.haeseler@uni-jena.de}

\author{Matthias Keller}
\address{Mathematisches Institut \\
Friedrich Schiller Universit{\"a}t Jena \\
07743 Jena, Germany }
\email{m.keller@uni-jena.de}

\author{Rados{\l}aw K. Wojciechowski}
\address{York College of the City University of New York \\
Jamaica, NY 11451 \\ USA  }
\email{rwojciechowski@gc.cuny.edu}


\begin{abstract}
We consider operators arising from regular Dirichlet forms with vanishing killing term. We give bounds for the bottom of the (essential) spectrum in terms of exponential volume growth with respect to an intrinsic metric. As special cases we discuss operators on graphs. When the volume growth is measured in the natural graph distance (which is not an intrinsic metric) we discuss the threshold for positivity of the bottom of the  spectrum and finiteness of the bottom of the essential spectrum of the (unbounded) graph Laplacian. This threshold is shown to lie at cubic polynomial growth.
\end{abstract}

\maketitle

\section{Introduction and Main Results}
In 1981 Brooks proved that the bottom of the essential spectrum of the Laplace Beltrami operator on a complete non compact Riemannian manifold with infinite measure can be bounded by the exponential volume growth rate of the manifold \cite{Br}.   Following this, similar results were proven in various contexts, see \cite{DK,Fuj,Hi,Hi2,OU,Stu}. Very recently it was shown in \cite{KLW} that such a result fails to be true in the case of the (non-normalized) graph Laplacian when the volume is measured with respect to the natural graph distance. Indeed, there are graphs of cubic polynomial volume growth that have positive bottom of the spectrum and slightly more than cubic growth already allows for purely discrete spectrum. This suggests that one should look for other candidates for a metric on a graph.

In this work we use the context of regular Dirichlet forms (without killing term) and the corresponding concept of  intrinsic metrics, see \cite{Stu} and \cite{FLW}, to prove a Brooks-type theorem. The purpose of this approach is threefold. First, we provide a set up which includes all known examples (and various others, e.g., quantum graphs) and give a unified treatment. Additionally, our estimates are slightly better than most of the previous results. Secondly, our method of proof seems to be much clearer and simpler than most of the previous works. Finally, graph Laplacians are now included and the disparity discussed above is resolved by considering suitable metrics. As an application, we can now prove that the examples found in \cite{KLW} for Laplacians on graphs do indeed give the borderline for positive bottom of the spectrum.  In particular, for the natural graph distance the threshold for zero bottom of the essential spectrum and the discreteness of the spectrum lies at cubic growth.

\smallskip
Let $X$ be a locally compact separable metric space and $m$ a positive Radon measure of full support.
Let $\E$ be a closed, symmetric, non-negative form on the Hilbert space $L^{2}(X,m)$ of real-valued square integrable functions with domain $D$.  We assume that $\E$ is a regular Dirichlet form without killing term (for background on Dirichlet forms see \cite{Fuk}, more details are given in Section~\ref{s:DF}). Let $L$ be the positive self adjoint operator arising from $\E$. Define
\begin{align*}
 \lm_{0}(L):=\inf\si(L)\quad\mbox{and}\quad   \lm_{0}^{\mathrm{ess}}(L):=\inf\se(L)
\end{align*}
where $\se(L)$ denotes the essential spectrum of $L$.

We let $\rho$ be an intrinsic pseudo metric in the sense of \cite{FLW}. For $x_{0}\in X$ and $r\geq 0$, we define the distance ball $B_{r}=B_{r}(x_{0})=\{x\in X\mid \rho(x,x_{0})\leq r\}$.
Let the \emph{exponential volume growth} be defined as
\begin{align*}
\mu=\liminf_{r\to\infty}\frac{1}{r}\log m(B_{r}(x_{0})).
\end{align*}
Note that, in contrast to previous works on manifolds \cite{Br}, graphs \cite{Fuj} and strongly local forms \cite{Stu}, we consider a $\liminf$ here, rather than a $\limsup$.

If $\rho$ takes values in $[0,\infty)$, then $X=\bigcup_{r} B_{r}(x_{0})$. In this case $\mu$ does not depend on the particular choice of $x_{0}$.
There is another constant first introduced in \cite{Stu} which we call the  \emph{minimal exponential volume growth} and which is defined as
\begin{align*}
\ow\mu=\liminf_{r\to\infty}\frac{1}{r}\inf_{x\in X}\log \frac{m(B_{r}(x))}{m(B_{1}(x))}.
\end{align*}

In this paper we prove the following theorem.

\begin{thm}\label{t:main} Let $L$ be the positive self adjoint operator arising from a regular Dirichlet form $\E$ without killing term and let $\rho$ be an intrinsic metric such that all distance balls are compact. Then,
\begin{align*}
\lm_{0}(L)\leq \frac{\ow{\mu}^{2}}{4}.
\end{align*}
If additionally $m(\bigcup_{r}B_{r}(x_{0}))=\infty$ for some $x_{0}$, then
\begin{align*}
    \lm_{0}^{\mathrm{ess}}(L)\leq \frac{\mu^{2}}{4}.
\end{align*}
\end{thm}

This has the following immediate corollary. The corollary has various  consequences, for example, the exponential instability of the semigroup $(e^{-tL})_{t\geq0}$ on $L^{p}(X,m)$, $p\in [1,\infty]$, see \cite[Corollary~2]{Stu}.

\begin{corollary} Suppose that $(X,d)$ is of subexponential growth, i.e., $\ow\mu=0$ (respectively, $\mu=0$). Then, $\lm_{0}(L)=0$ (respectively, $\lm_{0}^{\mathrm{ess}}(L)=0$).
\end{corollary}

\begin{remark}(a) Let us discuss Theorem~\ref{t:main} in the perspective of the present literature: For the Laplace Beltrami operator on a Riemannian manifolds an estimate for $\lm_{0}^{\mathrm{ess}}$ can be found in \cite{Br}, see also \cite{Hi2}. In \cite{Stu} the statement for $\lm_{0}$ is proven for strongly local Dirichlet forms. For non-local operators such results were known only for normalized Laplacians on graphs, see \cite{DK,Fuj,Hi,OU}.  These operators are of a very special form, in particular, they are always bounded. For unbounded Laplacians on graphs the conclusions of the theorem do not hold if one considers volume with respect to the natural graph metric, see \cite{KLW}. However, by \cite{FLW} (see also \cite{GHM}), there is  now a suitable notion of intrinsic metric for non-local forms. Let us stress that our result covers the results in \cite{Br,DK,Fuj,OU,Stu}. The results of type \cite{Hi,Hi2} could certainly also be obtained with slightly more technical effort which we avoid here for clarity of presentation.

(b) Despite the fact that our result is much more general, we have a unified method of proof for the bounds on the spectrum and the essential spectrum. Moreover, for  the essential spectrum, the proof is significantly simpler than the one of \cite{Br,Fuj} as we use test functions that converge weakly to zero and, therefore, avoid a cut-off procedure.

(c) Indeed, we prove a slightly more general result than above for non-local forms in Section~\ref{s:nonlocal}. In particular, for some special cases we prove much better estimates and recover the results of \cite{DK,Fuj,OU} in Corollary~\ref{c:normalized} in  Section~\ref{s:graph}.

(d) If we assume that $\rho$ takes values in $[0,\infty)$, then we can clearly replace the assumption that $m(\bigcup_{r}B_{r}(x_{0}))=\infty$ with $m(X)=\infty$.  The case when $m(X) < \infty$ is notably different, see \cite{HKLW2} for more details.

(e) If $\inf_{x\in X}m(B_{1}(x))>0$, then one can also show that $ \lm_{0}^{\mathrm{ess}}(L)\leq {\ow\mu^{2}}/{4}$.

(f) Our result deals exclusively with Dirichlet forms with vanishing killing term. The major challenge in the case of non vanishing killing term is to give a proper definition of volume which incorporates the killing term. We shortly discuss a strategy of how one could approach this case:  We  need  an positive generalized harmonic function $u$, i.e., $\E(u,\ph)=0$ for all $\ph\in D$, where $u$ is assumed to be  locally in the domain of $\E$ (this space is introduced in \cite{FLW} as $\mathcal{D}_{\mathrm{loc}}^{*}$). Such a function exists in many settings, see e.g. \cite{DK,HK, LSV}, and the result which guarantees the existence of such a function is often referred to as a Allegretto-Piepenbrink type theorem.  Then, by a ground state representation, see Theorem~10.1 \cite{FLW}, one obtains a form $\E_{u}$ with vanishing killing term such that $\E=\E_{u}$ on the intersection of their domains. Now, we can apply the methods above for $\E_{u}$ to derive the result for $\E$. However, as shown in \cite{HK}, there are examples of non-locally finite weighted graphs that do not have such a generalized harmonic function. Therefore, it would be interesting to find sufficient conditions under which the approach above can be carried out.
\end{remark}

Let us highlight one of the applications of our results for graphs. Let $\Delta $ be the graph Laplacian on $\ell^{2}(X)$ acting as
\begin{align*}
\Delta\ph(x)=    \sum_{y \sim x}(\ph(x)-\ph(y))
\end{align*}
(for more details, see Sections~\ref{s:graph} and~\ref{s:graph2}). Moreover, let $B_{r}^{d}$, for $r\geq0$, be balls with respect to the natural graph distance $d$ defined as the length of the shortest path of edges between two vertices.  It has to be stressed that this metric is not an intrinsic metric for $\Delta$. However, we will show in Theorem~\ref{t:graph2} that, if the growth of the balls $B_{r}^{d}$ is $r^{3-\eps}$ for any $\eps>0$, then $\lm_{0}(\Delta)=\lm_{0}^{\mathrm{ess}}(\Delta)=0$ and if it is less than $r^{3}$, then $\lm_{0}^{\mathrm{ess}}(\Delta)<\infty$. We demonstrate by examples that this result is sharp, see Section~\ref{s:graph2}.

\smallskip

The paper is structured as follows. In Section~\ref{s:Preliminaries} we recall some basic facts about Dirichlet forms and intrinsic metrics. Moreover, we give a bound on the bottom of the essential spectrum via weak null sequences and introduce the test functions. In Section~\ref{s:proof} we prove the crucial estimate for the strongly local and the non-local parts of the Dirichlet form and prove the main theorem. In Section~\ref{s:applications} we discuss the result for weighted graphs and prove the polynomial growth bound discussed above. 

Note added: After this work was completed we learned about the very recent preprint of Matthew Folz ``Volume growth and spectrum for general graph Laplacians" which contains related material in the special case of graphs.

\section{Preliminaries}\label{s:Preliminaries}
In this section we introduce the basic notions and concepts. The first subsection is devoted to recalling the setting of Dirichlet forms. In the second subsection we prove an estimate for the bottom of the essential spectrum and in the third subsection we discuss the basic properties of the test functions that are used to prove our result.

\subsection{Dirichlet forms}\label{s:DF}

In this section we recall some elementary facts about Dirichlet forms, see e.g. \cite{Fuk} and, for recent work on non-local forms, \cite{FLW}.

As above let $X$ be a locally compact separable metric space and let $m$  be a positive Radon measure of full support. We consider all functions on $X$ to be real-valued, but, by complexifying the corresponding Hilbert spaces and forms, we could also consider complex-valued functions.  A closed non-negative form on $L^{2}(X,m)$ consists of a dense subspace $D\subseteq L^{2}(X,m)$ and a sesqui-linear non-negative map $\E:D\times D\to\R$ such that $D$ is complete with respect to the form norm $\|\cdot\|_{\E}=\sqrt{\E(\cdot,\cdot)+\|\cdot\|^{2}}$ where $\|\cdot\|$ always denotes the $L^{2}$ norm.
We write $\E(u):=\E(u,u)$ for $u\in D$.

A closed non-negative form $(\E,D)$ is called a\emph{ Dirichlet form }if for any $u\in D$ and any normal contraction  $c:\R\to\R$ we have $c\circ u\in D$ and $\E(c\circ u)\leq \E(u)$. Here, $c$ is a normal contraction if $c(0)=0$ and $|c(x)-c(y)|\leq|x-y|$ for $x,y\in\R$. A Dirichlet form is called \emph{regular} if $D\cap C_{c}(X)$ is dense both in $(D,\|\cdot\|_{\E})$ and $(C_{c}(X),\|\cdot\|_{\infty})$ where $C_{c}(X)$ is the space of continuous compactly supported functions.

A function $f:X\to\R$ is said to be \emph{quasi continuous} if for every $\eps>0$ there is an open set $U\subseteq X$ with
\begin{align*}
    \mathrm{cap}(U):=\inf\{\|v\|_{\E}\mid v\in D,\, 1_{U}\leq v\}\leq \eps,
\end{align*}
 such that $f\vert _{X\setminus U}$ is continuous (where $\inf\emptyset=\infty$ and $1_{U}$ is the characteristic function of $U$). For a regular Dirichlet form $(\E,D)$ every $u\in D$ admits a quasi continuous representative, see \cite[Theorem 2.1.3]{Fuk}. In the following we assume that when considering $u$ as a function we always choose a quasi continuous representative.

There is a fundamental representation theorem for regular Dirichlet forms called the Beurling-Deny formula, see \cite[Theorem 3.2.1.]{Fuk}.  It states that there is a non-negative Radon measure $k$ on $X$, a non-negative Radon measure $J$ on $X\times X\setminus d$ which is $X\times X$ without the diagonal $d:=\{(x,x)\mid x\in X\}$ and a  positive semi-definite bilinear form $\Gm^{(c)}$ on $D\times D$ with values in the signed Radon measures on $X$ which is\emph{ strongly local}, i.e., satisfies $\Gm^{(c)}(u,v)=0$ if $u$ is constant on the support of $v$, such that
\begin{align*}
    \E(u)=\int_{X}d\Gm^{(c)}(u)+\int_{X\times X\setminus d} (u(x)-u(y))^{2}dJ(x,y)+\int_{X}u(x)^{2}dk(x),
\end{align*}
where we choose a quasi continuous representative of $u$ in the second and third integral. The first term on the right hand side is called the \emph{strongly local part} of $\E$, the second term is called the \emph{jump part} and the third term is called the \emph{killing term}.
The measure $J$ gives rise to a Radon measure $\Gm^{(j)}$ (where the $j$ refers to `jump') which is characterized by
\begin{align*}
    \int_{K}d\Gm^{(j)}(u)=\int_{K\times X\setminus d}(u(x)-u(y))^{2}dJ(x,y)
\end{align*}
for $K\subseteq X$ compact and $u\in D$. The focus of this paper is on regular Dirichlet forms $\E$ without killing term, i.e., $k\equiv0$. Thus, we denote
\begin{align*}
    \Gm=\Gm^{(c)}+\Gm^{(j)}.
\end{align*}

The space $D_{\mathrm{loc}}^{*}$ of \emph{functions locally in the domain} of $\E$ was introduced in \cite{FLW} and is important for the definition of intrinsic metrics. It is defined as the set of functions $u\in L^{2}_{\mathrm{loc}}(X,m)$ such that for all open and relatively compact sets $G$ there is a function $v\in D$ such that $u$ and $v$ agree on $G$ and for all compact $K\subseteq X$
\begin{align*}
    \int_{K\times X\setminus d}(u(x)-u(y))^{2}dJ(x,y)<\infty.
\end{align*}
We can extend  $\Gm^{(c)}$ and $\Gm^{(j)}$ to $D_{\mathrm{loc}}^{*}$, see \cite[Remarks after the proof of Theorem~3.2.1.]{Fuk} and \cite[Proposition~3.3]{FLW}.

For the strongly local part we have 
a \emph{chain rule} (see \cite[Theorem 3.2.2.]{Fuk}) as follows: for $\ph:\R\to\R$ continuously differentiable with bounded derivative $\ph'$,
\begin{align*}
    \Gm^{(c)}(\ph(u),v) = \ph'(u)\Gm^{(c)}(u,v),\quad u,v\in     D_{\mathrm{loc}}^{*}\cap L^{\infty}(X,m).
\end{align*}

A \emph{pseudo metric} is a map $\rho:X\times X\to[0,\infty]$ which is symmetric, satisfies the triangle inequality and $\rho(x,x)=0$ for all $x\in X$. For $A\subseteq X$ we define  the map $\rho_{A}:X\to[0,\infty]$ by
\begin{align*}
    \rho_{A}(x)=\inf_{y\in A}\rho(x,y).
\end{align*}
If $\rho$ is a pseudo metric and $T>0$, then $\rho\wedge T$ is a pseudo metric and we have that $(\rho\wedge T)_{A}=\rho_{A} \wedge T$ and $|\rho_{A}(x)\wedge T-\rho_{A}(y)\wedge T|\leq\rho(x,y)$.

By \cite[Definition~4.1.]{FLW} a pseudo metric $\rho$ is called an\emph{ intrinsic metric} for the Dirichlet form $\E$ if there are Radon measures $m^{(c)}$ and $m^{(j)}$ with  $m^{(c)}+m^{(j)}\leq m$ such that for all $A\subseteq X$ and all $T>0$ the functions $\rho_{A}\wedge T$ are in $D_{\mathrm{loc}}^{*}\cap C(X)$ and satisfy
\begin{align*}
    \Gm^{(c)}(\rho_{A}\wedge T)\leq m^{(c)}\quad\mbox{and}\quad\Gm^{(j)}(\rho_{A}\wedge T)\leq m^{(j)}.
\end{align*}
This implies that if $A\subseteq X$ is such that $\rho_{A}(x)<\infty$ for all $x\in X$, then
$\rho_{A}\in D_{\mathrm{loc}}^{*}\cap C(X)$ and $\Gm(\rho_{A})\leq m$. We assume that $\rho$ is continuous with respect to the original topology.


\subsection{An estimate for the bottom of the essential spectrum}
The following Persson-type theorem seems to be standard in some settings, see \cite{Per, Gri}. However, since we are not able to find a proper reference in the literature which covers our case, we include a short proof.

\begin{proposition}\label{p:h} 
Let $h$ be a closed quadratic form on $L^{2}(X,m)$ that is bounded from below and let $H$ be the corresponding self adjoint operator.  Assume that there is a normalized sequence $(f_{n})$ in $D(h)$ that converges weakly to zero. Then,
\begin{align*}
\lm_{0}^{\mathrm{ess}}(H)\leq\liminf_{n\to\infty} h(f_{n}).
\end{align*}
\end{proposition}

\begin{proof} Without loss of generality assume that $h\geq0$ and that $\lm_0^{\mathrm{ess}}(H)>0$.
Let $0< \lm<\lm_{0}^{\mathrm{ess}}(H)$. We will show that there is an $N\geq0$ such that $h(f_{n})>\lm$ for all $n \geq N$. Let $\lm_{1}$ be such that $\lm<\lm_{1}<\lm_{0}^{\mathrm{ess}}(H)$ and let $\eps>0$ be arbitrary. Since $D(H)$ is a core for $D(h)$ there exist $g_{n}\in D(H)$ for all $n\geq0$ such that $\aV{f_{n}-g_{n}}_{h}^{2}=h(f_{n}-g_{n})+\aV{f_{n}-g_{n}}^{2}\leq \eps$ and $(g_{n})$ converges weakly to zero as well.
As $\lm_{1}<\lm_{0}^{\mathrm{ess}}(H)$, the spectral projection $E_{(-\infty,\lm_{1}]}$ of $H$ and the interval $(-\infty,\lm_{1}]$ is a finite rank operator.
Therefore, as $(g_{n})$ converges weakly to zero, there is an $N\geq0$ such that $\aV{E_{(-\infty,\lm_{1}]}g_{n}}^{2}< \eps$ for $n\geq N$.
Letting $\nu_{n}$ be the spectral measure of $H$ with respect to $g_{n}$, we estimate for $n\ge N$
\begin{align*}
h(g_{n})\geq\int_{\lm_{1}}^{\infty}td\nu_{n}(t)\geq
\lm_{1} \int_{\lm_{1}}^{\infty}d\nu_{n}(t)=\lm_{1}(\|g_{n}\|^{2}
-\|E_{(-\infty,\lm_{1}]}g_{n}\|^{2})> \lm_{1}(1-\eps),
\end{align*}
where we used $\lm_{1}\ge0$ as $h\geq0$.
Since $h(f_{n})\geq h(g_{n})-\eps$ by the choice of $g_{n}$, we conclude the asserted inequality by choosing
$\eps=(\lm_1-\lm)/(1+\lm_{1})>0$.
\end{proof}

\subsection{The test functions}
In this section we introduce the sequence of test functions which we will use to estimate the bottom of the (essential) spectrum.

If $\mu = \infty$ or $\ow{\mu} =\infty$ the statements of our theorem become obvious, therefore, from now on, we assume that $\mu, \ow{\mu} < \infty$.

For $r\in\N,x_{0}\in X,\al>0$, define
\begin{align*}
f_{r,x_{0},\al}&:X\to[0,\infty),\quad x\mapsto  \big( (e^{\al r}\wedge e^{\al (2r-\rho(x_{0},x))}) - 1 \big)\vee 0.
\end{align*}
Then, for fixed $r$, $\al$, $x_{0}$, we have $f\vert_{B_{r}}\equiv e^{\al r}-1$, $f\vert_{B_{2r}\setminus B_{r}}= e^{ \al(2r- \rho(x_{0},\cdot))}-1$ and $f\vert_{X\setminus B_{2r}}\equiv0$. Clearly, $f$ is spherically homogeneous, i.e., there exists $h:[0,\infty)\to[0,\infty)$ such that $f(x)=h(\rho(x_{0},x))$.
The definition of $f$ combines ideas from \cite{Br}, \cite{Fuj} and \cite{Stu}.

Moreover, for $r\in\N,x_{0}\in X,\al>0$, let $g_{r,x_{0},\al}:X\to[0,\infty)$, be given by
\begin{align*}
g_{r,x_{0},\al}=(f_{r,x_{0},\al}+2)1_{B_{2r}}
\end{align*}


\begin{lemma}\label{l:f} Let $\al>\mu/2$, $x_{0}\in X $ and $f_{r}=f_{r,x_{0},\al}$ and $g_{r}=g_{r,x_{0},\al}$ for $r\geq0$. Then,
\begin{itemize}
  \item [(a)]  $f_{r},g_{r}\in L^{2}(X,m)$ for all $r\ge0$.
  \item [(b)] If $m(\bigcup_{r}B_{r})=\infty$, then $f_{r}/\|f_{r}\|$ converges weakly to $0$ as $r\to\infty$.
   \item [(c)] There is a sequence $(r_{k})$ such that $\|g_{r_{k}}\|/\|f_{r_{k}}\|\to 1$  as $k\to\infty$.
\end{itemize}
If $\al>\ow\mu/2$, then
\begin{itemize}
   \item [(d)] There are a sequences $(x_{k})$ in $X$ and $(r_{k})$ such that $f_{k}=f_{r_{k},x_{k},\al},g_{k}=g_{r_{k},x_{k},\al}\in L^{2}(X,m)$ and we have that $\|g_{k}\|/\|f_{k}\|\to 1$  as $k\to\infty$.
\end{itemize}
\end{lemma}
\begin{proof}
(a) As $\mu<\infty$ it follows that $m(B_{r}(x_{0}))<\infty$ for all $r\geq0$.  Therefore, $f_{r},g_{r}\in L^{2}(X,m)$ for all $r\ge0$ since $f_{r},g_{r}$ are supported in $B_{2r}$ and bounded.\\
(b) Let $\psi\in L^{2}(X,m)$ with $\aV{\psi}=1$, $\eps>0$ and set $\ph=\psi1_{\bigcup B_{r}}$. There exists $R>0$  such that $\aV{\ph 1_{X\setminus B_{R}}}\leq \eps/2$. Moreover, let $r\geq R$ be such that $m(B_{R})\leq\eps^{2} m(B_{r})/4$ (this choice is possible since $m(\bigcup B_{r})=\infty$).
We conclude by the Cauchy-Schwarz inequality  and $\|f_{r}1_{B_{R}}\|\leq\frac{\eps}{2}\|f_{r}\|$ that
\begin{align*}
\as{\ph,f_{r}} =\as{\ph 1_{B_{R}},f_{r}} + \as{\ph 1_{X\setminus B_{R}},f_{r}}\leq  \aV{\ph}\aV{f_{r}1_{B_{R}}}+\aV{\ph 1_{X\setminus B_{R}}} \aV{f_{r}}\leq \eps \aV{f_{r}}.
\end{align*}
As $\supp f_{r}\subseteq \bigcup_{s} B_{s}$,  it follows that $\as{\psi,f_{r}}=\as{\ph,f_{r}}$ for $r\geq0$ which proves (b).\\
Before we prove (c) we show (d) and indicate how to adapt the proof to (c) afterwards. 
Let $0<\eps<\al-\ow\mu/2$. By the definition of $\ow \mu$ there are sequences $(r_{k})$ of increasing positive numbers and $(x_{k})$ of elements in $X$ such that
\begin{align*}
    \frac{ m(B_{2r_{k}}(x_{k}))}{m(B_{1}(x_{k}))}&\leq e^{(2\ow \mu+\eps)r_{k}},\quad k\geq0.
\end{align*}
We set $f_{k}=f_{r_{k},x_{k},\al}$, $g_{k}=g_{r_{k},x_{k},\al}$. As $m(B_{2r_{k}}(x_{k}))<\infty$ and the functions $f_{k},g_{k}$ are supported in $B_{2r_{k}}(x_{k})$ and bounded, they are in $L^{2}(X,m)$. By definition we have
$g_{k}=g_{k}1_{B_{2r_{k}}}=(f_{k}+2)1_{B_{2r_{k}}}$, $k\geq0$. Using the inequalities $(a+b)^{2}\leq \frac{1}{(1-\eps)}a^{2}+\frac{1}{\eps}b^{2}$ and $\|f_{k}\|^{2}\geq  m({B_{r_{k}}}(x_{k}))(e^{\al r_{k}}-1)^{2}\geq m({B_{r_{k}}}(x_{k}))e^{2\al r_{k}}/c$ for some $c>0$   we get
\begin{align*}
\frac{\|g_{k}\|^{2}}{\|f_{k}\|^{2}}
&\leq\frac{(\|f_{k}\|+2\sqrt{m(B_{2r_{k}}(x_{k}))})^{2}}{\|f_{k}\|^{2}}
\leq\frac{\frac{1}{(1-\eps)}\|f_{k}\|^{2}+\frac{4}{\eps}{m({B_{2r_{k}}(x_{k})})}} {\|f_{k}\|^{2}}\\
&\leq\frac{1}{(1-\eps)}+\frac{4c}{\eps} {\frac{{m({B_{2r_{k}}(x_{k})})}}{m({B_{r_{k}}}(x_{k}))}} e^{-2\al r_{k}}.
\end{align*}
For $r_{k}$ large enough we have
\begin{align*}
\frac{ m(B_{r_{k}}(x_{k}))}{m(B_{1}(x_{k}))}\geq\inf_{x\in X}\frac{ m(B_{r_{k}}(x))}{m(B_{1}(x))}\geq e^{(\ow \mu-\eps)r_{k}}.
\end{align*}
Thus, by the choice of $(r_{k})$ and $(x_{k})$, we have ${\frac{{m({B_{2r_{k}}})}}{m({B_{r_{k}}})}} \leq e^{(\ow\mu+2\eps)r_{k}}$. As $0<\eps<\al-\ow\mu/2$
\begin{align*}
\frac{\|g_{{k}}\|^{2}}{\|f_{{k}}\|^{2}}\leq\frac{1}{(1-\eps)}+\frac{4c}{\eps}
e^{(\ow\mu+2\eps-2\al) r_{k}}\to \frac{1}{(1-\eps)}\quad\mbox{as $k\to\infty$}.
\end{align*}
Since $\eps$ can be chosen to be arbitrarily small and ${\|g_{{k}}\|}\geq{\|f_{{k}}\|}$ we deduce the statement.\\
For (c) we choose $(x_{k})$ to be $x_{0}$ and follow the lines of the proof replacing $\ow\mu$ by $\mu$.
\end{proof}

\begin{remark}If  $\inf_{x\in X}m(B_{1}(x))>0$, then $f_{k}/\|f_{k}\|$ of (d) also converges weakly to zero as $k\to\infty$.
\end{remark}

The following auxiliary estimates will later give us bounds for the Lipshitz constants of $f_{r,x,\al}$.

\begin{lemma}\label{l:e} Let $\al>0$. For all $R\geq 0$  one has
\begin{align*}
\frac{{\left({e}^{\alpha R} -1\right)}^{2}}{\left({e}^{2\alpha R} +1\right)}
\leq\frac{{\alpha }^{2}{R}^{2}}{2}.
\end{align*}
 Moreover, for $ R\in [0, 1] $ one has
\begin{align*}\frac{{\left({e}^{\alpha R} -1\right)}^{2}}{\left({e}^{2\alpha R} +1\right)}
\leq\frac{{R}^{2} {\left({e}^{\alpha}-1\right)}^{2}}{\left( R^2{e}^{2\alpha}+1\right)}.
\end{align*}
\end{lemma}\begin{proof}
For the first statement let $s=\al R$ and check via a series expansion that $s\mapsto{s}^{2}\left({e}^{2s} +1\right)- 2{\left({e}^{s}-1\right)}^{2}$ is non-negative.
The second statement follows by direct calculation since we have $e^{\al R}-1\leq R(e^{\al}-1)$ for $R\in[0,1]$ and $\al > 0$.
\end{proof}

\begin{lemma}\label{l:f_Lip} Let $r\in\N$, $x_{0}\in X$, $\al>0$ and set $f:=f_{r,x_{0},\al}$, $g:=g_{r,x_{0},\al}$.
Then, for all $ x,y\in X$
\begin{align*}
(f(x)-f(y))^{2}&\leq c(\al) (g(x)^{2}+g(y)^2) \rho(x,y)^{2}
\end{align*}
where $c(\al)= \frac{{\al^{2}}}{2} $.
If additionally  $\rho(x,y)\leq 1$, then $c(\al)$ can be chosen to be $c(\al,\rho(x,y))=\frac{(e^{\al}-1)^{2}}{\rho(x,y)^{2}e^{2\al}+1}$.
In particular, $f$ is Lipshitz continuous with Lipshitz constant $\al( e^{\al r}+1)$.
\end{lemma}
\begin{proof} We fix $r$, $\al$ and $x_{0}$ for the proof.
Let $x,y\in X$ be given and let $s=\rho(x_{0},x)$ and $t=\rho(x_{0},y)$.
We define $    D_{s,t}:=(f(x)- f(y))^{2}$.
Moreover, we use the estimate on $F(R):=\frac{(e^{\al R}-1)^{2}}{e^{2\al R}+1}$, $R\geq0,$  by $c(\al)R^{2}$ (and by $c(\al,R)R^{2}$ for $R\leq1$) from Lemma~\ref{l:e}.
By symmetry we  may assume, without loss of generality, that $s\leq t$ so that we have six cases to check.\\
Case 1: If $s\leq t\leq r$, then $D_{s,t}=0$.\\
Case 2: If $s\leq r\leq t\leq 2r$, then since $t-r\leq t-s=\rho(x_{0},y)-\rho(x_{0},x)\leq \rho(x,y)$ and $g(x)=e^{\al r}+1$, $g(y)=e^{\al (2r-t)}+1$,
\begin{align*}
D_{s,t}&=(e^{\al r} -e^{\al (2r-t)})^{2} =( e^{2\al r} +e^{2\al (2r-t)})F(t-r)
\leq ( e^{2\al r} +e^{2\al (2r-t)})c(\al)(t-r)^{2}\\
&\leq  c(\al)(g(x)^{2}+ g(y)^{2})\rho(x,y)^{2}.
\end{align*}
Case 3: If $s\leq r\leq 2r\leq t$, then since $r\leq t-s\leq \rho(x,y)$, $g(x)=e^{\al r}+1$ and $g({y})=0$,
\begin{align*}
D_{s,t}=(e^{\al r} -1)^{2} = (e^{2\al r}+1)F(r) \leq (e^{2\al r}+1)c(\al) r^{2}\leq2c(\al) (g(x)^{2}+g(y)^{2})\rho(x,y)^{2}.
\end{align*}
Case 4:  If $ r\leq s\leq t\leq 2r$, then since $t-s\leq \rho(x,y)$ and  $g(x)=e^{\al(2r- s)}+1$, $g(y)=e^{\al (2r-t)}+1$,
\begin{align*}
D_{s,t}&=(e^{\al(2r-s)} -e^{\al(2r-t)})^{2}= ( e^{2\al(2r- s)}+ e^{2\al(2r -t)}) F(t-s)\\
&\leq  c(\al)(g(x)^{2}+ g(y)^{2})\rho(x,y)^{2}.
\end{align*}
Case 5:  If $ r\leq s\leq 2r\leq t$, then since $2r-s\leq t-s\leq \rho(x,y)$, $g(x)=e^{\al(2r- s)}+1$ and $g(y)=0$,
\begin{align*}
D_{s,t}=(e^{\al(2r-s)}-1)^{2} =
( e^{2\al (2r-s)}+1)F(2r-s)\leq c(\al)(g(x)^{2}+g(y)^{2})\rho(x,y)^{2}.
\end{align*}
Case 6: If $2r\leq s\leq t$, then $D_{s,t}=0$.\\
The Lipshitz bound follows since $g$ is bounded by $e^{\al r}+1$.
\end{proof}

\begin{lemma}\label{l:f_in_D} Let $(\E,D)$ be a regular Dirichlet form and $\rho$ an intrinsic metric. For all  $r>0$, $x_{0}\in X$ and $\al>0$ we have $f:=f_{r,x_{0},\al}\in D_{\mathrm{loc}}^{*}$. Moreover, if $B_{2r}(x_{0})$ is compact, then $f\in D$.
\end{lemma}
\begin{proof}
By Lemma~\ref{l:f_Lip} the functions $f:=f_{r,x_{0},\al}$ are Lipshitz continuous for all $r>0$, $x_{0}$ and $\al>0$. Thus, by a Rademacher type theorem, see e.g. \cite[Theorem~5.1]{Sto} for strongly local forms or \cite[Theorem~4.8]{FLW} for general Dirichlet forms, we have $f\in D_{\mathrm{loc}}^{*}$ and $\Gm(f)\leq m$. If $B_{2r}(x_{0})$ is compact, then the function $f$ is compactly supported which implies that $f\in D$.
\end{proof}

\section{Proof of the main theorem}\label{s:proof}
\subsection{The strongly local estimate}
In this subsection we give an estimate which will be used to prove the theorem for the strongly local part of the Dirichlet form. For given $r\in\N$, $x_{0}\in X$  and $\al>0$ we denote $f:=f_{r,x_{0},\al}$ and $g:=g_{r,x_{0},\al}$.

\begin{lemma}\label{l:SL} Let $\rho$ be an intrinsic metric for a regular strongly local Dirichlet form $\mathcal{E}$. Then, for all $r>0$, $x_{0}\in X$ and $\al>0$ such that $f\in D$ we have
\begin{align*}
\E(f)\leq\al^{2}\int_{X}g^{2}dm^{(c)}.
\end{align*}
\end{lemma}
\begin{proof}
As $\E$ is strongly local, we get by the chain rule and the fact that $\rho $ is an intrinsic metric that
\begin{align*}
\E(f)&=\int_{B_{2r}\setminus B_{r}}d\Gm^{(c)}(f)=
\int_{B_{2r}\setminus B_{r}}d\Gm^{(c)}(e^{\al(2r-\rho(x_{0},\cdot))} -1)\\
&=\al^{2}\int_{B_{2r}\setminus B_{r}} e^{2\al(2r-\rho(x_{0},\cdot))}d\Gm^{(c)}(\rho(x_{0},\cdot))\\
&\leq\al^{2}\int_{B_{2r}\setminus B_{r}} e^{2\al(2r-\rho(x_{0},\cdot))}dm^{(c)}\leq \al^{2}\int_{X}g_{r,x_{0},\al}^{2}dm^{(c)}.
\end{align*}
\end{proof}

\subsection{The non-local estimate}\label{s:nonlocal}
Next, we treat the non-local case.  With applications to graphs in the next section in mind, we do not assume that the jump part is a regular Dirichlet form for now.

For this subsection, let $m$ be a Radon measure on $X$ and let $J$ be a symmetric Radon measure on $X\times X\setminus d$ such that for every $m$-measurable $A\subseteq X$ the set $A\times X\setminus d$ is $J$ measurable and vice versa. Let $\rho$ be a pseudo metric on $X$ which is $J$ measurable and assume that for all measurable $A\subseteq X$
\begin{align}\label{e:adapted}\tag{$\clubsuit$}
\int_{A\times X\setminus d}\rho(x,y)^{2}dJ(x,y)\leq m(A)
\end{align}
which immediately implies that for all  measurable functions $\ph$
\begin{align*}
\int_{X\times X\setminus d}\ph(x)^{2}\rho(x,y)^{2}dJ(x,y)\leq \int_{X} \ph^{2} dm.
\end{align*}
We say that the pseudo metric $\rho$ has \emph{jump size in} $[a,b]$, $0\leq a\leq b$, if for the set $A_{a,b}:=\{(x,y)\in X \times X \mid\rho(x,y)\in[a,b]\}\setminus d$
\begin{align*}
\int_{X\times X\setminus d}\rho(x,y)^{2}dJ(x,y)=\int_{A_{a,b}} \rho(x,y)^{2}dJ(x,y).
\end{align*}
For given $r\in\N$, $x_{0}\in X$  and $\al>0$ we denote $f:=f_{r,x_{0},\al}$ and $g:=g_{r,x_{0},\al}$.
\medskip

\begin{lemma}\label{l:NL}Assume that $\rho$ satisfies \eqref{e:adapted}.
For all $r\in\N$, $x_{0}\in X$  and $\al>0$
\begin{align*}
\int_{X\times X\setminus d}(f(x)-f(y))^{2}dJ(x,y)\leq\ 2c({\al }) \int_{X} g^2dm,
\end{align*}
where $c(\al)= \frac{\al^{2}}{2}$.
If $\rho$ has jump size in $[\de,1]$ for some $0\leq\de\leq1$, then $c(\al)$ can be chosen to be $c(\al,\de)=\frac{(e^{\al}-1)^{2}}{1+\de^{2}e^{2\al}}$.
\end{lemma}
\begin{proof}
By Lemma~\ref{l:f_Lip} and since $\rho$ satisfies \eqref{e:adapted}
\begin{align*}
\int_{X\times X\setminus d} (f(x)-f(y))^{2}dJ(x,y)\leq
{\al^{2}}\int_{X\times X\setminus d} g(x)^{2}\rho(x,y)^{2}dJ(x,y)\leq{\al^{2}} \int_{X} g^2dm.
\end{align*}
Let $\de>0$.
If the jump size is in $[\de,1]$, then
\begin{align*}
\int_{X\times X\setminus d} &(f(x)-f(y))^{2}dJ(x,y)= \int_{{A_{\de,1}}} (f(x)-f(y))^{2} dJ(x,y)\\
\leq&
\int_{{A_{\de,1}}} |g(x)|^{2}\frac{2(e^{\al}-1)^{2}}{(1+\rho(x,y)^{2}e^{2\al})} \rho(x,y)^{2}dJ(x,y) \\ \leq& \frac{2(e^{\al}-1)^{2}}{(1+\de^{2}e^{2\al})} \int_{X\times X\setminus d} g(x)^{2}\rho(x,y)^{2}dJ(x,y)\\
\leq& \frac{2(e^{\al}-1)^{2}}{(1+\de^{2}e^{2\al})} \int_{X}g^{2}dm.
\end{align*}
\end{proof}

\subsection{Proof of Theorem~\ref{t:main}}
We  now have all of the ingredients to prove our main result.\medskip

\begin{proof}[Proof of Theorem~\ref{t:main}] By \cite[Lemma~4.7]{FLW} an intrinsic metric satisfies \eqref{e:adapted}. Moreover, under the assumption that the distance balls are compact we have that $f_{r,x,\al}\in D$ for all $r>0$, $x\in X$, $\al>0$ by Lemma~\ref{l:f_in_D}.

By Lemma~\ref{l:f}~(d) there are a sequences $(x_{k})$ and $r_{k}$ such that for $f_{k}=f_{r_{k},x_{k},\al}$, $g_{k}=g_{r_{k},x_{k},\al}$ with $\al>\ow\mu/2$
\begin{align*}
\lm_{0}(L)\leq  \lim_{k\to\infty}\frac{\E(f_{k})}{\|f_{k}\|^{2}}\leq \al^{2}\lim_{k\to\infty}\frac{\|g_{k}\|^{2}}{\|f_{k}\|^{2}}=\al^{2}, \end{align*}
where  the second inequality follows from Lemmas~\ref{l:SL} and~\ref{l:NL} and the equality follows from Lemma~\ref{l:f}~(d). Hence, $\lm_{0}(L)\leq \ow \mu^{2}/4$. Let now $(r_{k})$ be the sequence given by Lemma~\ref{l:f}~(c)  for some fixed $x_{0}\in X$ and let $x_{k}=x_{0}$ for all $k\geq0$. By Lemma~\ref{l:f}~(b) the sequence $(f_{k}/\|f_{k}\|)$ converges weakly to zero and, therefore, we get by Proposition~\ref{p:h} and Lemma~\ref{l:f}~(c), that
\begin{align*} \lm_{0}^{\mathrm{ess}}(L) \leq\lim_{k\to\infty}\frac{\E(f_{k})}{\|f_{k}\|^{2}} \leq \al^{2}\lim_{k\to\infty}\frac{\|g_{k}\|^{2}}{\|f_{k}\|^{2}}=\al^{2}. \end{align*}
Therefore, $\lm_{0}^{\mathrm{ess}}(L) \leq\mu^{2}/4$.
\end{proof}

\subsection{A more general non-local estimate}
Let $L$ be the positive selfadjoint operator associated to $\E$.\medskip

\begin{thm}\label{t:jump} Assume that $\rho$ satisfies \eqref{e:adapted} and  $f_{r,x,\al}\in D$ for all $r\geq0$, $x\in X$ and $\al>\ow\mu/2$. Then,
\begin{align*}    \lm_{0}(L)\leq\frac{\ow\mu^{2}}{4}\quad\mbox{and}\quad \lm_{0}^{\mathrm{ess}}(L)\leq \frac{\mu^{2}}{4}
\end{align*}
if $m(\bigcup B_{r}(x_{0}))=\infty$ for $x_{0}$ used to define $\mu$.\\
If the jump size is bounded in $[\de,1]$ for some $0\leq\de\leq1$, then
\begin{align*}    \lm_{0}(L)\leq\frac{2{(e^{\ow\mu/2}-1)^{2}}}{{\de^{2}e^{\ow\mu}}+1} \quad\mbox{and}\quad \lm_{0}^{\mathrm{ess}}(L)\leq \frac{2{(e^{\mu/2}-1)^{2}}}{{\de^{2}e^{\mu}+1}}
\end{align*}
if $m(\bigcup B_{r}(x_{0}))=\infty$ for $x_{0}$ used to define $\mu$.
\end{thm}
\begin{proof}The proof follows analogously to the proof of the main theorem from  Proposition~\ref{p:h}, Lemma~\ref{l:f} and Lemma~\ref{l:NL}.
\end{proof}

\section{Applications}\label{s:applications}

\subsection{Weighted graphs}\label{s:graph}

In this section we derive consequences of Theorem~\ref{t:main} and Theorem~\ref{t:jump} for graphs. We briefly introduce the setting and refer for more background to \cite{KL1}.

Let $X$ be a countable discrete set. Every Radon measure of full support on $X$ is given by a function $m:X\to(0,\infty)$. Then, $L^{2}(X,m)$ is the space $\ell^{2}(X,m)$ of $m$-square summable functions with norm $\aV{u}=(\sum_{x}u(x)^{2}m(x))^{\frac{1}{2}}$, $u\in \ell^{2}(X,m)$.
From \cite[Theorem~7]{KL1} it can be seen that all regular Dirichlet forms without killing term are determined by a symmetric map $b:X\times X\to[0,\infty)$ with vanishing diagonal that satisfies
\begin{align*}
\sum_{y\in X}b(x,y)<\infty,\qquad\mbox{for all } x\in X,
\end{align*}
which gives rise to a measure $J$ on $X\times X\setminus d$ by $J=\frac{1}{2}b$. The one half stems from the convention that in the form we consider each edge only once.

The map $b$ can then be interpreted as a weighted graph with vertex set $X$. Namely, the vertices $x,y\in X$ are connected by an edge with weight $b(x,y)$ if $b(x,y)>0$. In this case, we write $x\sim y$. A graph is called \emph{connected} if for all $x,y\in X$ there are vertices $x_i \in X$ such that $x=x_{0}\sim x_{1}\sim\ldots \sim x_{n}=y$.

Let a map $\ow \E:\ell^{2}(X,m)\to[0,\infty]$ be given by
\begin{align*}
    \ow \E(u)=\frac{1}{2}\sum_{x,y\in X}b(x,y)(u(x)-u(y))^{2}.
\end{align*}
The regular Dirichlet form $\E$ associated to $J$ is the restriction of $\ow \E$ to $\overline{C_{c}(X)}^{\|\cdot\|_{\mathcal{E}}}$.
Moreover, let
\begin{align*}
\mbox{ $\E^{\max}=\ow \E\vert_{D^{\max}},\quad D^{\max}=\{u\in \ell^{2}(X,m)\mid \ow \E(u)<\infty\}$}\end{align*}
which is also a Dirichlet form.
We denote the operator arising from  $\E$ by $L$ and the operator arising from $\E^{\max}$ by $L^{\max}$.

Let $\rho$ be an intrinsic pseudo metric on $X$. In this context this is equivalent to  \eqref{e:adapted} (see \cite[Lemma~4.7, Theorem~7.3]{FLW}) which reads as
\begin{align*}
\frac{1}{2}\sum_{y\in X}b(x,y)\rho(x,y)^{2}\leq m(x),\quad x\in X.
\end{align*}

For simplicity we restrict ourselves to the case when $\rho$ takes values in $[0,\infty)$. (Otherwise, we can easily consider the graph componentwise.)

\begin{remark}\label{r:rho} Very often it is convenient to consider intrinsic metrics which satisfy
$\sum_{y\in X}b(x,y)\rho(x,y)^{2}\leq m(x)$ for all $ x\in X$ (i.e., we drop the $\frac{1}{2}$ on the left hand side). For example, in \cite{Hu} an explicit example of such a metric $\rho$ is given, for $x,y \in X$, by
\begin{align*}     \rho(x,y):=\inf\{l(x_{0},\ldots,x_{n})\mid n\geq 1,x_{0}=x, x_{n}=y, x_{i}\sim x_{i-1}, i=1,\ldots,n\}
\end{align*}
where the length $l$ is given by $l(x_{0},\ldots,x_{n})= \sum_{i=1}^{n}\min\{\Deg(x_{i})^{-\frac{1}{2}}, \Deg(x_{i-1})^{-\frac{1}{2}}\}$ and $\Deg(z)=\sum_{w}b(z,w)/m(z)$
is a generalized vertex degree. In this case all estimates in the theorem above can be divided by $2$.
\end{remark}

In general, it is hard to determine whether distance balls with respect to a certain metric are compact, which means finite in the original topology, in the situation of graphs. However, we always have a statement for the operator $L^{\max}$ related to $\E^{\max}$.\medskip

\begin{thm}\label{t:graph}Assume that $b$ is connected and $m(X) = \infty$.
Then,
\begin{align*}    \lm_{0}(L^{\max})\leq\frac{\ow\mu^{2}}{4}\quad\mbox{and}\quad \lm_{0}^{\mathrm{ess}}(L^{\max})\leq \frac{\mu^{2}}{4}.
\end{align*}
If $\rho(x,y)\in[\de,1]$ for all $x\sim y$, then
\begin{align*}    \lm_{0}(L^{\max})\leq\frac{{2(e^{\ow\mu/2}-1)^{2}}}{{\de^{2}e^{\ow\mu}}+1} \quad\mbox{and}\quad \lm_{0}^{\mathrm{ess}}(L^{\max})\leq \frac{{2(e^{\mu/2}-1)^{2}}}{{\de^{2}e^{\mu}+1}}.
\end{align*}
\end{thm}

\begin{remark}
In this case where the assumption on the adapted metric above is posed without the $1/2$ on the left hand side, all estimates in the theorem above can be divided by $2$.
\end{remark}

\begin{proof} Let $B_{r_{k}}(x_{0})$, (respectively $B_{\ow r_{k}}(x_{k})$) be a sequence of distance balls that realizes $\mu$ (respectively $\ow \mu$), i.e., $\mu=\lim_{k\to\infty} r_{k}^{-1}\log m(B_{r_{k}}(x_{0}))$ (respectively $\ow\mu=\lim_{k\to\infty} r_{k}^{-1}\log m(B_{\ow r_{k}}(x_{k}))$). If the measure of  $B_{r_{k}}(x_{0})$, (respectively $B_{\ow r_{k}}(x_{k})$) is infinite for some $k$, then $\mu=\infty$ (respectively $\ow\mu=\infty$) and we are done. Otherwise, $f_{r_{k},x_{0},\al},g_{r_{k},x_{0},\al} \in \ell^{2}(X,m)$ (respectively  $f_{\ow r_{k},x_{k},\al}$, $g_{\ow r_{k},x_{k},\al}\in \ell^{2}(X,m)$) and $f_{r_{k},x_{0},\al}\in  D^{\max}$ (respectively  $f_{\ow r_{k},x_{k},\al}\in D^{\max}$) by Lemma~\ref{l:NL}. Thus, the statement follows directly from Theorem~\ref{t:jump}.
\end{proof}

In the case when we know more about the measure or the metric structure we can say something about the operator $L$.  This is the case under either of the following additional assumptions:
\begin{itemize}
  \item [(A)] Every infinite path of vertices has infinite measure.
  \item [(B)] $\rho$ is any adapted path metric on a locally finite graph such that $(X, \rho)$ is metrically complete.
\end{itemize}
In particular, (A) is satisfied if $\inf_{x\in X}m(x)>0$ and (B) is satisfied if all infinite geodesics have infinite length.

 \medskip

\begin{corollary}\label{c:graph}Assume that either (A) or (B) is satisfied. Then, the statement of Theorem~\ref{t:graph} holds for $L=L^{\max}$.
\end{corollary}
\begin{proof}
By \cite[Theorem 6]{KL1}, respectively \cite[Theorem 2]{HKMW}, (A), respectively (B), imply that $\E=\E^{\max}$ and $L=L^{\max}$.
\end{proof}

\begin{remark} Under the slightly stronger assumption that connected infinite sets have infinite measure we can prove the corollary directly. Namely, if one of the relevant distance balls is infinite, then it has infinite measure and the exponential volume growth is infinite. In the other case the corollary follows from Theorem~\ref{t:jump}.
\end{remark}

We also recover the result of \cite{Fuj} which already covers \cite{DK,OU}. In their very particular situation, $m$ is the vertex degree and $b$ takes values in $\{0,1\}$. The natural graph distance $d$ is given as the minimum length of a path of edges connecting two vertices where the length is the number of edges contained in the path.
\medskip

\begin{corollary}(Normalized Laplacians)\label{c:normalized} Let $b$ be a connected weighted graph over $(X,n)$, with $n(x)=\sum_{y\in X}b(x,y)$, $x\in X$ and  let $d$ be the natural graph metric. Then, $\lm_{0}^{\mathrm{ess}}(L)\leq 1- 2e^{\ow\mu/2}/({1+e^{\ow\mu}})$ and
$\lm_{0}(L)\leq 1-2e^{\mu/2}/({1+e^{\mu}})$.
\end{corollary}
\begin{proof} Clearly, $L$ is a bounded operator and thus $L=L^{\max}$. Moreover,  the natural graph metric is an intrinsic metric for $2L$ and its jump size in exactly $1$. Thus, the statement follows from the previous theorem.
\end{proof}

\subsection{Unweighted graphs and the natural graph distance}\label{s:graph2}

Let $b:X\times X\to\{0,1\}$ and $m\equiv 1$. Then, the operator $L$ becomes the graph Laplacian $\Delta $ acting on $D(\Delta)=\{\ph\in \ell^{2}(X)\mid (x\mapsto\sum_{y\sim x}(\ph(x)-\ph(y)))\in\ell^{2}(X)\}$, see \cite{KL1,Woj1}, as
\begin{align*}
    \Delta\ph(x)=\sum_{y\sim x}(\ph(x)-\ph(y)),
\end{align*}
where $x\sim y$ means that $b(x,y)=1$. By $m\equiv 1$ we have that $m(A)=|A|$ for all $A\subseteq X$. For simplicity we assume that the graph is connected.

\medskip

\begin{thm} \label{t:graph2} Let the $d$ be the natural graph distance on an infinite graph and $B_{r}^{d}=\{x\in X\mid d(x,x_{0})\leq r\}$ for some $x_{0}\in X$ and $r\ge0$. If
\begin{align*}
\liminf_{r\to\infty}\frac{\log |B_{r}^{d}(x_{0})|}{\log r}<3,
\end{align*}
then, $\lm_{0}(\Delta)=\lm_{0}^{\mathrm{ess}}(\Delta)=0$. Moreover, if
\begin{align*}
\limsup_{r\to\infty}\frac{ |B_{r}^{d}(x_{0})|}{ r^{3}}<\infty,
\end{align*}
 then $\lm_{0}^{\mathrm{ess}}(\Delta)<\infty$ and, in particular, $\si_{\mathrm{ess}}(\Delta)\neq \emptyset$.
\end{thm}
\begin{remark}(a) The result above is sharp. This can be seen by the examples of antitrees discussed below the proof.

(b) In \cite[Theorem~1.4]{GHM} it is shown that less than cubic growth implies stochastic completeness.

(c) In the case where the vertex degree is bounded by some $K$, the situation is very different: the $n$ in Corollary~\ref{c:normalized} becomes $\deg$ in our situation, where $\deg:X\to\N$ is the function assigning to a vertex the number of adjacent vertices, and the corresponding normalized operator is $\ow\Delta$ acting on $\ell^{2}(X,\deg)$ as
$\ow\Delta\ph(x)=\frac{1}{\deg(x)}\sum_{y\sim x}(\ph(x)-\ph(y))$. Then,
\begin{align*}
  \lm_{0}(\ow \Delta)\leq\lm_{0}(\Delta)\leq K\lm(\ow\Delta)\quad\mbox{and}\quad    \lm_{0}^{\mathrm{ess}}(\ow \Delta)\leq\lm_{0}^{\mathrm{ess}}(\Delta)\leq K\lm^{\mathrm{ess}}(\ow\Delta),
\end{align*}
see, e.g., \cite{K}. Thus, in the bounded situation, the threshold  lies again at subexponential growth  by Corollary~\ref{c:normalized} (as the measures $m\equiv 1$ and $n=\deg$ also give the same exponential volume growth.) Explicit estimates for the exponential volume growth of planar tessellations in terms of curvature can be found in \cite{KP}.

(d) In the case of bounded vertex degree we also have a threshold for recurrence of the corresponding random walk at quadratic volume growth, see \cite[Lemma~3.12]{Woe}.
\end{remark}

Let $\rho$ be the intrinsic metric from \cite{Hu} introduced above in Remark~\ref{r:rho} which, in the case of unweighted graphs, is given by
$$\rho(x,y)=\inf\{
\sum_{i=0}^{n-1}\min\{\deg(x_{i})^{-\frac{1}{2}}, \deg(x_{i+1})^{-\frac{1}{2}}\}\mid (x_{0},\ldots, x_{n})\mbox{ is a path from $x$ to $y$}\}.$$
Let $ B_{r}^{\rho}=\{x\in X\mid \rho(x,x_{0})\leq r\}$, while $B_{r}^{d}$ are the balls with respect to the natural graph distance $d$.

The proof of the theorem is based on the following lemma which  is inspired by the proof of \cite[Theorem~1.4]{GHM}. Indeed, the second statement is taken directly from there.\medskip

\begin{lemma} If $\liminf\limits_{r\to\infty}{\log |B_{r}^{d}|}/{\log r}=\beta\in[1,3)$,
then $\liminf\limits_{r\to\infty}{\log|B_{r}^{\rho}|}/{\log r}\leq {\frac{2\beta}{3-\beta}}$. Moreover, if $\limsup\limits_{r\to\infty}{ |B_{r}^{d}|}/{ r^{3}}<\infty$, then $\limsup\limits_{r\to\infty}\frac{1}{r}\log|B_{r}^{\rho}|<\infty$.
\end{lemma}
\begin{proof}Let $S_{r}^{d}=B_{r}^{d}\setminus B_{r-1}^{d}$, $r\geq0$,  and for convenience set $S_{-r}^{d}=B_{-r}^{d}=\emptyset$ for $r>0$.
Let $1\leq\al<3$ and $(r_{k})$ be an increasing sequence such that ${\log |B_{r_{k}}^{d}(x_{0})|}/{\log r_{k}}<\al$ for all $k\ge0$.  Then,
\begin{align*}
    |B_{r_{k}}^{d}|=\sum_{r=0}^{r_{k}}|S_{r}^{d}|< r^{\al}_{k}
\end{align*}
for large $k\ge0$. For $\eps>0$ and $k\geq0$ set
\begin{align*}
A_{k}:=\{r\in[0,r_{k}]\cap \N_{0}\mid |S_{r}^{d}|> \frac{\al}{\eps^{\al}} r^{\al-1}\}.
\end{align*}
We can estimate $|A_{k}|\leq\eps r_{k}$ via
\begin{align*}
r_{k}^{\al}>|B_{r_{k}}^{d}|
\geq\frac{\al}{\eps^{\al}}\sum_{r\in A_{k}}r^{\al-1}
\geq\frac{\al}{\eps^{\al}}\sum_{r=0}^{|A_{k}|}r^{\al-1} \geq\frac{\al}{\eps^{\al}}\int_{0}^{|A_{k}|}r^{\al-1}dr
=\frac{|A_{k}|^{\al}}{\eps^{\al}}.
\end{align*}
Thus,
\begin{align*}
    |\{r\in[1,r_{k}]\cap \N_{0}\mid \max_{i=0,1,2,3}|S_{r-i}^{d}|> \frac{\al}{\eps^{\al}} r^{\al-1}\}|\leq 4\eps r_{k}
\end{align*}
and
\begin{align*}
    |\{r\in[1,r_{k}]\cap \N_{0}\mid \max_{i=0,1,2,3}|S_{r-i}^{d}|\leq \frac{\al}{\eps^{\al}} r^{\al-1}\}|\geq (1-4\eps)r_{k}.
\end{align*}
As we have $\deg \leq |S_{r-1}^{d}\cup S_{r}^{d}\cup S_{r+1}^{d}|$ on $S_{r}^{d}$, we get $|D_{k}|\geq (1-4\eps)r_{k}$, where
\begin{align*}
D_{k}:=\{(r+1)\in[0,r_{k}-1]\cap \N_{0}\mid \deg\leq \frac{3\al}{\eps^{\al}} r^{\al-1}\mbox{ on } S_{r-1}^{d}\cup S_{r}^{d}\}.
\end{align*}
Hence, for $(r+1)\in D_{k}$ we have for $x\in S_{r-1}^{d}$, $y\in S_{r}^{d}$
$$\rho(x,y)\geq  cr^{-\frac{\al-1}{2}},\quad\mbox{ with } c=\sqrt{{\eps^{\al}}/{3\al}}.$$ Since any path from $x_{0}$ to  $S_{r_{k}}^{d}$ contains such edges we have
for any $x\in  S_{r_{k}}^{d}$
\begin{align*}
\rho(x_{0},x)\geq c\sum_{(r+1)\in D_{k}}r^{-\frac{\al-1}{2}}
\geq c \sum_{r=4\eps r_{k}}^{r_{k}-1}r^{-\frac{\al-1}{2}}
\geq c\int_{4\eps r_{k}} ^{r_{k}-1}r^{-\frac{\al-1}{2}} dr\geq C_{0} r_{k}^{{\frac{3-\al}{2}}}
\end{align*}
with $C_{0}>0$ for $\eps>0$ chosen sufficiently small and $r_{k}$ large. Let   $R_{k}:=C_{0} r_{k}^{{\frac{3-\al}{2}}}$ and $C:=C_{0}^{-\frac{2\al}{3-\al}}$. Then, $B_{R_{k}}^{\rho}\subseteq B_{r_{k}}^{d}$ and since $|B_{r_{k}}^{d}|=\sum_{r=0}^{r_{k}}|S_{r}^{d}|< r^{\al}_{k}$, we conclude
\begin{align*}
|B_{R_{k}}^{\rho}|\leq |B_{r_{k}}^{d}|<r^{\al}_{k}\leq CR_{k}^{\frac{2\al}{3-\al}}.
\end{align*}
Thus, the first statement follows. The second statement is shown in the proof of \cite[Theorem~1.4]{GHM}.
\end{proof}

\begin{proof}[Proof of Theorem~\ref{t:graph2}] In the case where the polynomial growth is strictly less than cubic we get by the lemma above that $\mu=0$ with respect to the intrinsic metric $\rho$ and in the case where it is less than cubic we still have $\mu<\infty$. Thus, the statement follows from Corollary~\ref{c:graph}, where (A) is clearly satisfied as $m\equiv 1$.
\end{proof}

Let us discuss the example of antitrees which show the sharpness of the result.  They were first introduced in \cite{Woj3} and further studied in \cite{BK,KLW}.

\begin{example}
An antitree is a spherically symmetric graph, where a vertex in the $r$-th sphere is connected to all vertices in the $(r+1)$-th sphere for $r\ge0$, and there are no horizontal edges. Thus, an antitree is characterized by a sequence $(s_{r})$ taking values in $\N$ which encodes the number of vertices in the sphere $S_{r}^{d}=B_{r}^{d}\setminus B_{r-1}^{d}$.

\emph{Stronger growth than cubic:} In \cite[Corollary~6.6]{KLW} it is shown that if the polynomial volume growth of an antitree is more than cubic, i.e., as $r^{3+\eps}$ for $\eps>0$, then $\lm_{0}(\Delta)>0$ and $\si_{\mathrm{ess}}(\Delta)=\emptyset$. Indeed, in the intrinsic metric $\rho$, these antitrees have finite diameter and thus $\mu=\infty$, see \cite{Hu}.

\emph{Cubic growth:} If the distance spheres of an antitree satisfy $|S_{r}^{d}|=(r+1)^{2}$, then $|B_{r}^{d}|\sim (r+1)^{3}$. Moreover, the function which takes the value $r^{-2}$ on vertices of the $(r-1)$-th sphere, $r\geq1$, is a positive generalized super-solution for $\Delta$ to the value $2$, that is, $\Delta \ph \geq 2 \ph$. Thus, by a discrete Allegretto-Piepenbrink theorem (see \cite[Theorem~4.1]{Woj2} or \cite[Theorem~3.1]{HK}) it follows that $\lm_{0}(\Delta)\geq2$. By Theorem~\ref{t:graph2} we thus have  $2\leq \lm_{0}^{\mathrm{ess}}(\Delta)<\infty$.

\emph{Weaker growth than cubic:} In this case Theorem~\ref{t:graph2} shows that $\lm_{0}(\Delta)=\lm_{0}^{\mathrm{ess}}(\Delta)=0$.
\end{example}



\textbf{Acknowledgements.}  The authors are grateful to J{\'o}zef Dodziuk and Daniel Lenz for their continued support and for generously sharing their knowledge.  The research of RKW was partially sponsored by the Funda\c{c}\~ao para a Ci\^encia e a Tecnologia through project PTDC/MAT/101007/2008 and by the Research Foundation of CUNY through the PSC-CUNY Research Award 42.


\end{document}